\documentclass[twoside,leqno,10pt, A4]{amsart}
\usepackage{amsfonts}
\usepackage{amsmath}
\usepackage{amscd}
\usepackage{amssymb}
\usepackage{amsthm}
\usepackage{amsrefs}
\usepackage{latexsym}
\usepackage{mathrsfs}
\usepackage{bbm}
\usepackage{enumerate}
\usepackage{graphicx}

\usepackage{amsfonts}
\usepackage{amsmath}
\usepackage{amscd}
\usepackage{amssymb}
\usepackage{amsthm}
\usepackage{amsrefs}
\usepackage{latexsym}
\usepackage{mathrsfs}
\usepackage{bbm}
\usepackage{amscd}
\usepackage{amssymb}
\usepackage{amsthm}
\usepackage{amsrefs}
\usepackage{latexsym}
\usepackage{mathrsfs}
\usepackage{bbm}
\usepackage{enumerate}
\usepackage{graphicx}
\usepackage{color}
\begin{document}

\newtheorem{theorem}[subsection]{Theorem}
\newtheorem{proposition}[subsection]{Proposition}
\newtheorem{lemma}[subsection]{Lemma}
\newtheorem{corollary}[subsection]{Corollary}
\newtheorem{conjecture}[subsection]{Conjecture}
\newtheorem{prop}[subsection]{Proposition}
\numberwithin{equation}{section}
\newcommand{\mr}{\ensuremath{\mathbb R}}
\newcommand{\mc}{\ensuremath{\mathbb C}}
\newcommand{\dif}{\mathrm{d}}
\newcommand{\intz}{\mathbb{Z}}
\newcommand{\ratq}{\mathbb{Q}}
\newcommand{\natn}{\mathbb{N}}
\newcommand{\comc}{\mathbb{C}}
\newcommand{\rear}{\mathbb{R}}
\newcommand{\prip}{\mathbb{P}}
\newcommand{\uph}{\mathbb{H}}
\newcommand{\fief}{\mathbb{F}}
\newcommand{\majorarc}{\mathfrak{M}}
\newcommand{\minorarc}{\mathfrak{m}}
\newcommand{\sings}{\mathfrak{S}}
\newcommand{\fA}{\ensuremath{\mathfrak A}}
\newcommand{\mn}{\ensuremath{\mathbb N}}
\newcommand{\mq}{\ensuremath{\mathbb Q}}
\newcommand{\half}{\tfrac{1}{2}}
\newcommand{\f}{f\times \chi}
\newcommand{\summ}{\mathop{{\sum}^{\star}}}
\newcommand{\chiq}{\chi \bmod q}
\newcommand{\chidb}{\chi \bmod db}
\newcommand{\chid}{\chi \bmod d}
\newcommand{\sym}{\text{sym}^2}
\newcommand{\hhalf}{\tfrac{1}{2}}
\newcommand{\sumstar}{\sideset{}{^*}\sum}
\newcommand{\sumprime}{\sideset{}{'}\sum}
\newcommand{\sumprimeprime}{\sideset{}{''}\sum}
\newcommand{\sumflat}{\sideset{}{^\flat}\sum}
\newcommand{\shortmod}{\ensuremath{\negthickspace \negthickspace \negthickspace \pmod}}
\newcommand{\V}{V\left(\frac{nm}{q^2}\right)}
\newcommand{\sumi}{\mathop{{\sum}^{\dagger}}}
\newcommand{\mz}{\ensuremath{\mathbb Z}}
\newcommand{\leg}[2]{\left(\frac{#1}{#2}\right)}
\newcommand{\muK}{\mu_{\omega}}
\newcommand{\thalf}{\tfrac12}
\newcommand{\lp}{\left(}
\newcommand{\rp}{\right)}
\newcommand{\Lam}{\Lambda_{[i]}}
\newcommand{\lam}{\lambda}
\def\L{\fracwithdelims}
\def\om{\omega}
\def\pbar{\overline{\psi}}
\def\phis{\phi^*}
\def\lam{\lambda}
\def\lbar{\overline{\lambda}}
\newcommand\Sum{\Cal S}
\def\Lam{\Lambda}
\newcommand{\sumtt}{\underset{(d,2)=1}{{\sum}^*}}
\newcommand{\sumt}{\underset{(d,2)=1}{\sum \nolimits^{*}} \widetilde w\left( \frac dX \right) }

\newcommand{\hf}{\tfrac{1}{2}}
\newcommand{\af}{\mathfrak{a}}
\newcommand{\Wf}{\mathcal{W}}

\theoremstyle{plain}
\newtheorem{conj}{Conjecture}
\newtheorem{remark}[subsection]{Remark}

\makeatletter
\def\widebreve{\mathpalette\wide@breve}
\def\wide@breve#1#2{\sbox\z@{$#1#2$}%
     \mathop{\vbox{\m@th\ialign{##\crcr
\kern0.08em\brevefill#1{0.8\wd\z@}\crcr\noalign{\nointerlineskip}%
                    $\hss#1#2\hss$\crcr}}}\limits}
\def\brevefill#1#2{$\m@th\sbox\tw@{$#1($}%
  \hss\resizebox{#2}{\wd\tw@}{\rotatebox[origin=c]{90}{\upshape(}}\hss$}
\makeatletter

\title[Bounds for moments of Dirichlet $L$-functions to a fixed modulus]{Bounds for moments of Dirichlet $L$-functions to a fixed modulus}

\author{Peng Gao}
\address{School of Mathematical Sciences, Beihang University, Beijing 100191, P. R. China}
\email{penggao@buaa.edu.cn}
\begin{abstract}
 We study the $2k$-th moment of central values of the family of Dirichlet $L$-functions to a fixed prime modulus. We establish sharp lower
 bounds for all real $k \geq 0$ and sharp upper bounds for $k$ in the range $0 \leq k \leq 1$.
\end{abstract}

\maketitle



\section{Introduction}
\label{sec 1}

  A considerable amount of work in the literature has been done on moments of central values of families of $L$-functions, due to rich
  arithmetic meanings these central values have. In this paper, we focus on the family of Dirichlet $L$-functions to a fixed modulus. It is
  widely believed that (see \cite{R&Sound}) for all real $k  \geq 0$ and large integers $q \not \equiv 2 \pmod 4$ (so that primitive Dirichlet
  characters modulo $q$ exist),
\begin{align}
\label{moments}
 \sumstar_{\substack{ \chi \shortmod q }}|L(\tfrac{1}{2},\chi)|^{2k} \sim C_k \phis(q)(\log q)^{k^2},
\end{align}
  where we denote $\chi$ (respectively, $\phis(q)$) for a Dirichelt character (respectively, the number of primitive characters) modulo $q$,
  the numbers $C_k$ are explicit constants and we denote throughout the paper $\sumstar$ for the sum over primitive Dirichlet characters
  modulo $q$.

  The formula given in \eqref{moments} is well-known for $k=1$ and is a conjecture due to K. Ramachandra \cite{Rama79} for $k=2$ when the sum
  in \eqref{moments} is being replaced by the sum over all Dirichlet characters modulo a prime $q$. For all most all $q$, D. R. Heath-Brown
  \cite{HB81} established \eqref{moments} for $k=2$ and K. Soundararajan \cite{Sound2007} improved the result to be valid for all $q$. An
  asymptotic formula with a power saving error term was further obtained in this case for $q$ being prime numbers by M. P. Young
  \cite{Young2011}. The main terms in Young's result agree with a conjectured formula provided by J. B. Conrey, D. W. Farmer, J. P. Keating,
  M. O. Rubinstein and N. C. Snaith in \cite{CFKRS} concerning the left side of \eqref{moments} for all positive integral values of $k$.
  Subsequent improvements on the error terms in Young's result are given in \cite{BFKMM1} and \cite{BFKMM}. See also \cite{Wu2020} for an
  extension of Young's result to general moduli.

   Other than the asymptotic relations given in \eqref{moments}, much is known on upper and lower bounds of the conjectured order of magnitude
   for moments of the family of $L$-functions under consideration. To give an account for the related results, we assume that $q$ is a prime
   number in the rest of the paper. In \cite{Sound01}, under the assumption of the generalized Riemann hypothesis (GRH), K. Soundararajan
   showed that
$$\sumstar_{\substack{ \chi \shortmod q }}|L(\tfrac{1}{2},\chi)|^{2k} \ll_k \phis(q)(\log q)^{k^2+\varepsilon} $$
  for all real positive $k$ and any $\varepsilon>0$. These bounds are optimal except for the $\varepsilon$ powers. The optimal upper bounds
  are later obtained by D. R. Heath-Brown in \cite{HB2010} unconditionally for $k=1/v$ with $v$ a positive integer and for all $k \in (0, 2)$
  under GRH. Using a sharpening of the method of Soundararajan by A. J. Harper in \cite{Harper}, one may also establish the optimal upper
  bounds for all real $k \geq 0$ under GRH. In \cite{Radziwill&Sound},  M. Radziwi{\l\l} and K. Soundararajan enunciated a principle that
  allows one to establish sharp upper bounds for moments of families of $L$-functions unconditionally and used it to study the moments of
  quadratic twists of $L$-functions attached to elliptic curves. This principle was then applied by W. Heap, M. Radziwi{\l\l} and K.
  Soundararajan in \cite{HRS} to establish unconditionally the $2k$-th moment of the Riemann zeta function on the critical line for all real
  $0 \leq k \leq 2$.

   In the opposite direction, a simple and powerful method developed by Z. Rudnick and K. Soundararajan in \cite{R&Sound1} shows that
$$\sumstar_{\substack{ \chi \shortmod q }}|L(\tfrac{1}{2},\chi)|^{2k} \gg_k \phis(q)(\log q)^{k^2} $$
  for all rational $k \geq 1$. A modification of a method of M. Radziwi{\l\l} and K. Soundararajan in \cite{Radziwill&Sound1}  may allow one
  to establish such lower bounds for all real $k \geq 1$. In \cite{C&L}, V. Chandee and X. Li obtained the above lower bounds for rational
  $0<k<1$.

  In  \cite{H&Sound}, W. Heap and K. Soundararajan developed another principle which allows one to study lower bounds of families of
  $L$-functions. This principle can be regarded as a companion to the above principle of M. Radziwi{\l\l} and K. Soundararajan
  \cite{Radziwill&Sound} concerning upper bounds. Although Heap and Soundararajan only studied moments of the Riemann zeta function on the
  critical line, they did point out that their principle may be applied to study moments of families of $L$-functions, including the one we
  consider in this paper. In fact,  the density conjecture of N. Katz and P. Sarnak concerning low-lying zeros of families of $L$-functions
  indicates that the underlying symmetry for the family of Dirichlet $L$-functions to a fixed modulus is unitary, and that the behaviour of
  this family resembles that of the Riemann zeta function on the critical line. Thus, one expects to obtain sharp lower bounds for moments of
  the above unitary family of $L$-functions using the principle of Heap and Soundararajan. The aim of this paper is to first carry out this
  principle explicitly to achieve the desired lower bounds in the following result.
\begin{theorem}
\label{thmlowerbound}
   For large prime $q$ and any real number $k \geq 0$, we have
\begin{align}
\label{lowerbound}
   \sumstar_{\substack{ \chi \shortmod q }}|L(\tfrac{1}{2},\chi)|^{2k} \gg_k \phis(q)(\log q)^{k^2}.
\end{align}
\end{theorem}

   Next, we apply the dual principle of M. Radziwi{\l\l} and K. Soundararajan \cite{Radziwill&Sound} to establish sharp upper bounds for a
   restricted range of $k$ as follows.
\begin{theorem}
\label{thmupperbound}
   For large prime $q$ and any real number $k$ such that $0 \leq k \leq 1$, we have
\begin{align}
\label{upperbound}
   \sumstar_{\substack{ \chi \shortmod q }}|L(\tfrac{1}{2},\chi)|^{2k} \ll_k \phis(q)(\log q)^{k^2}.
\end{align}
\end{theorem}

   Note that we can combine Theorem \ref{thmlowerbound} and \ref{thmupperbound} together to obtain the following result concerning the order
   of magnitude of our family of $L$-functions.
\begin{theorem}
\label{thmorderofmag}
   For large prime $q$ and any real number $k$ such that $0 \leq k \leq 1$, we have
\begin{align}
\label{orderofmag}
   \sumstar_{\substack{ \chi \shortmod q }}|L(\tfrac{1}{2},\chi)|^{2k} \asymp \phis(q)(\log q)^{k^2}.
\end{align}
\end{theorem}

   We notice that such a result in \eqref{orderofmag} is already implied by the above mentioned result of D. R. Heath-Brown \cite{HB2010} and
   that of  V. Chandee and X. Li \cite{C&L}. In particular, the case $k=1$ of \eqref{orderofmag} is explicitly given in \cite{C&L}. Moreover,
   the result is shown to be valid for $k=3/2$ as well by H. M. Bui, K. Pratt, N. Robles and A. Zaharescu \cite[Theorem 1.4]{BPRZ}. This case
   is achieved by employing various tools including a result on a long mollified second moment of the corresponding family of $L$-functions
   given in \cite[Theorem 1.1]{BPRZ}. In our proofs of Theorems \ref{thmlowerbound} and \ref{thmupperbound}, we also need to evaluate certain
   twisted second moments for the same family. The lengths of the corresponding Dirichlet polynomials are however short so that the main
   contributions come only from the diagonal terms. Hence, only the orthogonality relation for characters is needed to complete our work. We
   also point out here that as it is mentioned in \cite{BPRZ} that one may apply the work of B. Hough \cite{Hough2016} or R. Zacharias
   \cite{Z2019} on twisted fourth moment for the family of Dirichlet $L$-functions modulo $q$ to obtain sharp upper bounds on all moments
   below the fourth. We decide to use the twisted second moment here to keep our exposition simple by observing that it is needed for
   obtaining both the lower bounds and the upper bounds.

\section{Preliminaries}
\label{sec 2}

  We include a few auxiliary results in this section. We also reserve the letter $p$ for a prime number in this paper and
we recall the following result from \cite[Lemma 2.2]{Gao2021-2}.
\begin{lemma}
\label{RS} Let $x \geq 2$. We have, for some constant $b$,
$$
\sum_{p\le x} \frac{1}{p} = \log \log x + b+ O\Big(\frac{1}{\log x}\Big).
$$
 Also, for any integer $j \geq 1$, we have
$$
\sum_{p\le x} \frac {(\log p)^j}{p} = \frac {(\log x)^j}{j} + O((\log x)^{j-1}).
$$
\end{lemma}

  Next, we note the following approximate functional equation for $|L(1/2, \chi)|^2$.
\begin{lemma}
\label{PropDirpoly}
  Let $\af=0$ or $1$ be given by $\chi(-1)=(-1)^{\af}$. We have
\begin{align}
\label{lsquareapprox}
|L(\half, \chi)|^2 = 2 \sum^{\infty}_{a, b=1} \frac{\chi(a) \overline{\chi}(b)}{\sqrt{ab}} \Wf_{\af} \left(\frac {\pi ab}{q}\right),
\end{align}
  where
$$ \Wf_{\af}(x) = \frac{1}{2\pi i} \int\limits_{(c)} \frac{\Gamma\left(\frac{1}{4} + \frac{s + \af}{2}\right)^2}{\Gamma\left(\frac{1}{4} +
\frac{\af}{2} \right)^2 } x^{-s} \> \frac{ds}{s}.$$
  Moreover, the function $\Wf_{\af}(x)$ is real valued and satisfies the bound that for any $c>0$,
\begin{align*}
 \Wf_{\af}(x)  \ll_c \min( 1 , x^{-c}).
\end{align*}
\end{lemma}
  The above lemma follows by combining equations \cite[(1.2)-(1.4)]{Sound2007} and Lemma 2 there, together with the observation that the
  property that $\Wf_{\af}(x)$ is real valued can be established similar to \cite[Lemma 2.1]{sound1}.

  The presence of $\Wf_{\af}(x)$ in the expression for $|L(1/2, \chi)|^2$ makes it natural to consider sums over odd and even characters
  separately when summing over $\chi$ modulo $q$. For this reason, we denote $\phi(q)$ for the Euler totient function and note the following
  orthogonal relations.
\begin{lemma}{\cite[Lemma 1]{C&L}}
\label{lem:sumoverevenodd} Let $\sum_{\chi}^{(e)}, \sum_{\chi}^{(o)}$ indicate the sum over non-trivial primitive even (respectively odd)
characters modulo $q$. Then
$$ {\sum_{\chi}}^{(e)} \chi(a) = \left\{ \begin{array}{ll} \frac{\phi(q) - 2}{2} & {\rm if} \ \ a \equiv \pm 1 \ ({\rm mod} \ q ) \\
-1 & {\rm if} \ \ a  \not\equiv \pm 1 \ ({\rm mod} \ q ) \ {\rm and} \ (a,q) = 1,  \end{array} \right. $$
and
$$ {\sum_{\chi}}^{(o)} \chi(a) = \left\{ \begin{array}{ll} \frac{\phi(q) }{2} & {\rm if} \ \ a \equiv 1 \ ({\rm mod} \ q ) \\
\frac{-\phi(q)}{2} & {\rm if} \ \ a \equiv -1 \ ({\rm mod} \ q ) \\
0 & {\rm if} \ \ a  \not\equiv \pm 1 \ ({\rm mod} \ q ) \ {\rm and} \ (a,q) = 1.  \end{array} \right.$$
\end{lemma}

\section{Outline of the Proofs}
\label{sec 2'}

  We may assume that $q$ is a large prime number and we note that in this case we have $\phis(q)=q-2$. As the case $k=1$ for both
  \eqref{lowerbound} and \eqref{upperbound} is known, we may assume in our proofs that $k \neq 1$ is a fixed positive real number and let $N,
  M$ be two large natural numbers depending on $k$ only and  and denote $\{ \ell_j \}_{1 \leq j \leq R}$ for a sequence of even natural
  numbers such that $\ell_1= 2\lceil N \log \log q\rceil$ and $\ell_{j+1} = 2 \lceil N \log \ell_j \rceil$ for $j \geq 1$, where $R$ is
  defined to the largest natural number satisfying $\ell_R >10^M$.  We may assume that $M$ is so chosen so that we have $\ell_{j} >
  \ell_{j+1}^2$ for all $1 \leq j \leq R-1$ and this further implies that we have
\begin{align}
\label{sumoverell}
  \sum^R_{j=1}\frac 1{\ell_j} \leq \frac 2{\ell_R}.
\end{align}

    We denote ${ P}_1$ for the set of odd primes not exceeding $q^{1/\ell_1^2}$ and
${ P_j}$ for the set of primes lying in the interval $(q^{1/\ell_{j-1}^2}, q^{1/\ell_j^2}]$ for $2\le j\le R$. For each $1 \leq j \leq R$, we
write
\begin{equation*}
{\mathcal P}_j(\chi) = \sum_{p\in P_j} \frac{1}{\sqrt{p}} \chi(p), \quad  {\mathcal Q}_j(\chi, k) =\Big (\frac{12 \max (1, k^2) {\mathcal
P}(\chi) }{\ell_j}\Big)^{r_k\ell_j},
\end{equation*}
  where we define $r_k=\lceil 1+1/k \rceil+1$ for $0<k<1$ and $r_k=\lceil k /(2k-1) \rceil+1$ for $k>1$. We further define ${\mathcal Q}_{R+1}(\chi, k)=1$.

  We define for any non-negative integer $\ell$ and any real number $x$,
\begin{equation*}
E_{\ell}(x) = \sum_{j=0}^{\ell} \frac{x^{j}}{j!}.
\end{equation*}
  Further, we define for each $1 \leq j \leq R$ and any real number $\alpha$,
\begin{align*}
{\mathcal N}_j(\chi, \alpha) = E_{\ell_j} (\alpha {\mathcal P}_j(\chi)), \quad \mathcal{N}(\chi, \alpha) = \prod_{j=1}^{R} {\mathcal
N}_j(\chi,\alpha).
\end{align*}

 Before we proceed to our discussions below, we would like to point out here without further notice that in the rest of the paper, when we use
 $\ll$ or the $O$-symbol to estimate various quantities needed, the implicit constants involved only depend on $k$ and are uniform with
 respect to $\chi$. We shall also make the convention that an empty product is defined to be $1$.

  We now present the needed versions in our setting of the lower bounds principle of W. Heap and K. Soundararajanand in \cite{H&Sound} and the
  upper bounds principle of M. Radziwi{\l\l} and K. Soundararajan in \cite{Radziwill&Sound} in the following two lemmas. We choose to state our results suitable for our proofs of Theorems \ref{thmlowerbound} and \ref{thmupperbound} only. One may easily adjust them to study moments for various other families of $L$-functions.
  
  Our first lemma corresponds to the lower bounds principle.
\begin{lemma}
\label{lem1}
 With notations as above. For $0<k<1$, we have
\begin{align}
\label{basiclowerbound}
\begin{split}
\sumstar_{\substack{ \chi \shortmod q }}L(\tfrac{1}{2},\chi)  \mathcal{N}(\chi, k-1) \mathcal{N}(\overline{\chi}, k)
 \ll & \Big ( \sumstar_{\substack{ \chi \shortmod q }}|L(\tfrac{1}{2},\chi)|^{2k} \Big )^{1/2}\Big ( \sumstar_{\substack{ \chi \shortmod q
 }}|L(\tfrac{1}{2},\chi)|^2 |\mathcal{N}(\chi, k-1)|^2  \Big)^{(1-k)/2} \\
 & \times \Big ( \sumstar_{\substack{ \chi \shortmod q }}   \prod^R_{j=1}\big ( |{\mathcal N}_j(\chi, k)|^2+ |{\mathcal Q}_j(\chi,k)|^2 \big )
 \Big)^{k/2}.
\end{split}
\end{align}
 For $k>1$, we have
\begin{align}
\label{basicboundkbig}
\begin{split}
 & \sumstar_{\substack{ \chi \shortmod q }}L(\tfrac{1}{2},\chi)  \mathcal{N}(\chi, k-1) \mathcal{N}(\overline{\chi}, k)
 \leq  \Big ( \sumstar_{\substack{ \chi \shortmod q }}|L(\tfrac{1}{2},\chi)|^{2k} \Big )^{\frac {1}{2k}}\Big ( \sumstar_{\substack{ \chi
 \shortmod q }} \prod^R_{j=1} \big ( |{\mathcal N}_j(\chi, k)|^2+ |{\mathcal Q}_j(\chi,k)|^2 \big ) \Big)^{\frac {2k-1}{2k}}.
\end{split}
\end{align}
  The implied constants in \eqref{basiclowerbound} and \eqref{basicboundkbig} depend on $k$ only.
\end{lemma}
\begin{proof}
   We assume $0<k<1$ first and apply H\"older's inequality to see that the left side of \eqref{basiclowerbound} is
\begin{align}
\label{basicbound0}
\begin{split}
 \leq & \Big ( \sumstar_{\substack{ \chi \shortmod q }}|L(\tfrac{1}{2},\chi)|^{2k} \Big )^{1/2}\Big ( \sumstar_{\substack{ \chi \shortmod q
 }}|L(\tfrac{1}{2},\chi) \mathcal{N}(\chi, k-1)|^2   \Big)^{(1-k)/2}\Big ( \sumstar_{\substack{ \chi \shortmod q }} |\mathcal{N}(\chi,
 k)|^{2/k}|\mathcal{N}(\chi, k-1)|^{2}  \Big)^{k/2}.
\end{split}
\end{align}

 Notice that we have for $|z| \le aK/10$ with $0<a \leq 1$,
\begin{align}
\label{Ebound}
\Big| \sum_{r=0}^K \frac{z^r}{r!} - e^z \Big| \le \frac{|a z|^{K}}{K!} \le \Big(\frac{a e}{10}\Big)^{K},
\end{align}
   where the last estimation above follows from the observation that
\begin{align}
\label{Stirling}
  (\frac ne)^n \leq n! \leq n(\frac ne)^n.
\end{align}

  We apply \eqref{Ebound} with $z=k{\mathcal P}_j(\chi), K=\ell_j$ and $a=k$ to see that when $|{\mathcal P}_j(\chi)| \le \ell_j/10$,
\begin{align*}
{\mathcal N}_j(\chi, k)=& \exp( k {\mathcal P}_j(\chi))\Big( 1+  O\Big(\exp( k |{\mathcal P}_j(\chi)|)\Big(\frac{k e}{10}\Big)^{\ell_j} \Big )
=  \exp( k {\mathcal P}_j(\chi))\Big( 1+  O\Big( ke^{-\ell_j} \Big )\Big ).
\end{align*}
  Similarly, we have
\begin{align*}
{\mathcal N}_j(\chi, k-1)= & \exp( (k-1) {\mathcal P}_j(\chi))\Big( 1+  O\Big(e^{-\ell_j} \Big ) \Big ).
\end{align*}

   The above estimations then allow us to see that when $|{\mathcal P}_j(\chi)| \le \ell_j/10$,
\begin{align}
\label{est1}
|{\mathcal N}_j(\chi, k)^{\frac {1}{k}} {\mathcal N}_j(\chi, k-1)|^{2}
&= \exp( 2k \Re ({\mathcal P}_j(\chi)))\Big( 1+ O\big( e^{-\ell_j} \big) \Big) = |{\mathcal N}_j(\chi, k)|^2 \Big( 1+ O\big(e^{-\ell_j} \big)
\Big).
\end{align}

  On the other hand, we notice that when $|{\mathcal P}_j(\chi)| \ge \ell_j/10$,
\begin{align*}
\begin{split}
|{\mathcal N}_j(\chi, k)| &\le \sum_{r=0}^{\ell_j} \frac{|{\mathcal P}_j(\chi)|^r}{r!} \le
|{\mathcal P}_j(\chi)|^{\ell_j} \sum_{r=0}^{\ell_j} \Big( \frac{10}{\ell_j}\Big)^{\ell_j-r} \frac{1}{r!}   \le \Big( \frac{12 |{\mathcal
P}_j(\chi)|}{\ell_j}\Big)^{\ell_j} .
\end{split}
\end{align*}
  Observe that the same bound above also holds for $|{\mathcal N}_j(\chi, k-1)|$. It follows from these estimations that when $|{\mathcal
  P}_j(\chi)| \ge \ell_j/10$, we have
\begin{align*}
|{\mathcal N}_j(\chi, k)^{\frac {1}{k}} {\mathcal N}_j(\chi, k-1)|^{2}
& \leq \Big( \frac{12 |{\mathcal P}_j(\chi)|}{\ell_j}\Big)^{2(1+1/k)\ell_j} \leq  |{\mathcal Q}_j(\chi, k)|^2.
\end{align*}
  Applying the above together with \eqref{basicbound0} and \eqref{est1} allows us to establish the estimation  given in
  \eqref{basiclowerbound}.

  It remains to consider the case $k>1$ and we apply H\"older's inequality again to see that the left side of \eqref{basicboundkbig} is
\begin{align}
\label{holderkbig}
\begin{split}
 \leq  \Big ( \sumstar_{\substack{ \chi \shortmod q }}|L(\tfrac{1}{2},\chi)|^{2k} \Big )^{\frac {1}{2k}}\Big ( \sumstar_{\substack{ \chi
 \shortmod q }} |\mathcal{N}(\chi, k)\mathcal{N}(\chi, k-1)|^{\frac {2k}{2k-1}}  \Big)^{\frac {2k-1}{2k}}.
\end{split}
\end{align}
  We apply \eqref{Ebound} this time with $z=k{\mathcal P}_j(\chi),K=\ell_j, a=1$ and arguing as above to see that when $|{\mathcal P}(\chi)|
  \le \ell_j/(10k)$, we have
\begin{align}
\label{prodNkbig}
\begin{split}
 |\mathcal{N}_j(\chi, k)\mathcal{N}_j(\chi, k-1)|^{\frac {2k}{2k-1}} = |{\mathcal N}_j(\chi, k)|^2 \Big( 1+ O\big(e^{-\ell_j} \big) \Big).
\end{split}
\end{align}
  Similarly, when $|{\mathcal P}_j(\chi)| \ge \ell_j/(10k)$, we have
\begin{align}
\label{prodNkbig1}
\begin{split}
 |\mathcal{N}_j(\chi, k)\mathcal{N}_j(\chi, k-1)|^{\frac {2k}{2k-1}} \leq  \Big( \frac{12k^2 |{\mathcal P}_j(\chi)|}{\ell_j}\Big)^{\frac {2k
 \ell_j}{2k-1}}
  \leq  |{\mathcal Q}_j(\chi, k)|^2.
\end{split}
\end{align}
  We then deduce the estimation given in \eqref{basicboundkbig} readily from \eqref{holderkbig}, \eqref{prodNkbig} and \eqref{prodNkbig1}.
  This completes the proof of the lemma.
\end{proof}

  Our next lemma corresponds to the upper bounds principle.  Instead of the form used for obtaining upper bounds given in \cite{Radziwill&Sound}, we decide to adapt one that resembles what is given in Lemma \ref{lem1} above and also derive it via a similar fashion. One may compare our next lemma with \cite[Proposition 3]{Radziwill&Sound} and \cite[Proposition 2.1]{HRS}.
\begin{lemma}
\label{lem2}
 With notations as above. We have for $0<k<1$,
\begin{align}
\label{basiclowerbound2}
\begin{split}
& \sumstar_{\substack{ \chi \shortmod q }}|L(\tfrac{1}{2},\chi)|^{2k} \\
 \ll & \Big ( \sumstar_{\substack{ \chi \shortmod q }}|L(\tfrac{1}{2},\chi)|^2 \sum^{R}_{v=0} \prod^v_{j=1}\Big ( |\mathcal{N}_j(\chi, k-1)|^2 \Big ) |{\mathcal
 Q}_{v+1}(\chi, k)|^{2}
 \Big)^{k} \Big ( \sumstar_{\substack{ \chi \shortmod q }}  \sum^{R}_{v=0}  \Big (\prod^v_{j=1}|\mathcal{N}_j(\chi, k)|^2\Big )|{\mathcal
 Q}_{v+1}(\chi, k)|^{2} \Big)^{1-k},
\end{split}
\end{align}
  where the implied constants depend on $k$ only.
\end{lemma}
\begin{proof}
  Note first that using arguments similar to those in the proof of Lemma \ref{lem1}, we have that when $|{\mathcal P}_j(\chi)| \le \ell_j/10$,
\begin{align}
\label{prodNlowerbound}
 |\mathcal{N}_j(\chi, k-1)|^{2k}|\mathcal{N}_j(\chi, k)|^{2(1-k)} \geq 1+ O\big(e^{-\ell_j} \big ),
\end{align}
  where the implied constants are uniformly bounded for all $j$.

  Now,  if there exists an integer $0 \leq v \leq R-1$ such that $| \mathcal{P}_j (\chi) | \leq \ell_j/10$ whenever $j \leq v$, but with $|
  \mathcal{P}_{v+1} (\chi) | > \ell_{v+1}/10$,  we deduce from the above and the observation that $|{\mathcal Q}_{v+1}(\chi, k)| \geq 1$ when
  $|{\mathcal P}_{v+1}(\chi)| \ge \ell_{v+1}/10$ that
\begin{align*}
 \Big ( \prod^v_{j=1}|\mathcal{N}_j(\chi, k-1)|^{2k}|\mathcal{N}_j(\chi, k)|^{2(1-k)} \Big )|{\mathcal Q}_{v+1}(\chi, k)|^2 \gg 1.
\end{align*}

  If no such $v$ exists, then we must have $| \mathcal{P}_j (\chi) | \leq \ell_j/10$ for all $1 \leq j \leq R$ so that the estimation
  \eqref{prodNlowerbound} is valid for all $j$ and we have
\begin{align*}
 \prod^R_{j=1}|\mathcal{N}_j(\chi, k-1)|^{2k}|\mathcal{N}_j(\chi, k)|^{2(1-k)} \gg 1.
\end{align*}

  In either case, we conclude that
\begin{align*}
 \Big (\sum^R_{v=0}\Big ( \prod^v_{j=1}|\mathcal{N}_j(\chi, k-1)|^2 \Big )|{\mathcal Q}_{v+1}(\chi, k)|^2 \Big )^{k}\Big (\sum^R_{v=0}\Big ( \prod^v_{j=1}|\mathcal{N}_j(\chi, k)|^2\Big ) |{\mathcal Q}_{v+1}(\chi, k)|^2 \Big )^{1-k}  \gg 1.
\end{align*}
   We then deduce from this that
\begin{align*}
 & \sumstar_{\substack{ \chi \shortmod q }} |L(\half, \chi)|^{2k} \\
 \ll & \sumstar_{\substack{ \chi \shortmod q }} |L(\half, \chi)|^{2k} \Big (\sum^R_{v=0}\Big ( \prod^v_{j=0}|\mathcal{N}_j(\chi, k-1)|^2 \Big )|{\mathcal Q}_{v+1}(\chi, k)|^2 \Big )^{k}
 \times \Big (\sum^R_{v=0}\Big (  \prod^v_{j=1}|\mathcal{N}_j(\chi, k)|^2\Big ) |{\mathcal Q}_{v+1}(\chi, k)|^2 \Big )^{1-k}.
\end{align*}
     Applying H\"older's inequality to the last expression above leads to the estimation given in \eqref{basiclowerbound2} and this completes the proof of the lemma.
\end{proof}

  In what follows, we may further assume that $0<k<1$ by noting that the case $k \geq 1$ of Theorem \ref{thmlowerbound} can be obtained using
  \eqref{basicboundkbig} in Lemma \ref{lem1} and applying the arguments in the paper. We then deduce from Lemma \ref{lem1} and Lemma \ref{lem2} that in order to prove Theorem \ref{thmlowerbound} and
   \ref{thmupperbound} for the case $0<k<1$, it suffices to establish the following three propositions.
\begin{proposition}
\label{Prop4} With notations as above, we have
\begin{align*}
\sumstar_{\substack{ \chi \shortmod q }}L(\tfrac{1}{2},\chi) \mathcal{N}(\overline{\chi}, k) \mathcal{N}(\chi, k-1) \gg \phis(q)(\log q)^{ k^2
} .
\end{align*}
\end{proposition}

\begin{proposition}
\label{Prop5} With notations as above, we have
\begin{align*}
\max \Big (  \sumstar_{\substack{ \chi \shortmod q }}|L(\tfrac{1}{2},\chi)\mathcal{N}(\chi, k-1)|^2, \sumstar_{\substack{ \chi \shortmod q }}|L(\tfrac{1}{2},\chi)|^2 \sum^{R}_{v=0}\Big (\prod^v_{j=1}|\mathcal{N}_j(\chi, k-1)|^{2}\Big ) |{\mathcal
 Q}_{v+1}(\chi, k)|^2 \Big )   \ll
\phis(q)(\log q)^{ k^2 }.
\end{align*}
\end{proposition}

\begin{proposition}
\label{Prop6} With notations as above, we have
\begin{align*}
\max \Big ( \sumstar_{\substack{ \chi \shortmod q }}\prod^R_{j=1}\big ( |{\mathcal N}_j(\chi, k)|^2+ |{\mathcal Q}_j(\chi,k)|^2 \big ),  \sumstar_{\substack{ \chi \shortmod q }} \sum^{R}_{v=0} \Big ( \prod^v_{j=1}|\mathcal{N}_j(\chi, k)|^{2}\Big )|{\mathcal
 Q}_{v+1}(\chi, k)|^2 \Big )   \ll \phis(q)(\log q)^{ k^2 }.
\end{align*}
\end{proposition}

   We shall prove the above propositions in the rest of the paper.

\section{Proof of Proposition \ref{Prop4}}
\label{sec 4}

    Denote $\Omega(n)$ for the number of distinct prime powers dividing $n$ and $w(n)$ for the multiplicative function such that
    $w(p^{\alpha}) = \alpha!$ for prime powers $p^{\alpha}$.  Let $b_j(n), 1 \leq j \leq R$ be functions such that $b_j(n)=1$ when $n$ is
    composed of at most $\ell_j$ primes, all from the interval $P_j$. Otherwise, we define $b_j(n)=0$. We use these notations to see that for
    any real number $\alpha$,
\begin{equation}
\label{5.1}
{\mathcal N}_j(\chi, \alpha) = \sum_{n_j} \frac{1}{\sqrt{n_j}} \frac{\alpha^{\Omega(n_j)}}{w(n_j)}  b_j(n_j) \chi(n_j), \quad 1\le j\le R.
\end{equation}
    Note that each ${\mathcal N}_j(\chi, \alpha)$ is a short Dirichlet polynomial since $b_j(n_j)=0$ unless $n_j \leq
    (q^{1/\ell_j^2})^{\ell_j}=q^{1/\ell_j}$. It follows from this that ${\mathcal N}(\chi, k)$ and ${\mathcal N}(\chi, k-1)$ are short Dirichlet
    polynomials whose lengths are both at most $q^{1/\ell_1+ \ldots +1/\ell_R} < q^{2/10^{M}}$ by \eqref{sumoverell}. Moreover, it is readily
    checked that we have for each $\chi$ modulo $q$ (including the case $\chi=\chi_0$, the principal character modulo $q$),
\begin{align}
\label{prodNbound}
 & {\mathcal N}(\chi, k){\mathcal N}(\chi, k-1) \ll q^{2(1/\ell_1+ \ldots +1/\ell_R)} < q^{4/10^{M}}.
\end{align}

    We note further that it is shown in \cite{R&Sound} that for $X \geq 1$,
\begin{align*}
 & L(\half, \chi)=\sum_{m \leq X}\frac {\chi(m)}{\sqrt{m}}+O(\frac {\sqrt{q}\log q}{\sqrt{X}}).
\end{align*}
  We deduce from the above that
\begin{align*}
& \sumstar_{\substack{ \chi \shortmod q }}L(\tfrac{1}{2},\chi) \mathcal{N}(\overline{\chi}, k) \mathcal{N}(\chi, k-1) \\
= & \sumstar_{\substack{ \chi \shortmod q }}\sum_{m \leq X}\frac {\chi(m)}{\sqrt{m}}\mathcal{N}(\overline{\chi}, k) \mathcal{N}(\chi,
k-1)+O(\frac {\sqrt{q}\log q}{\sqrt{X}}\sumstar_{\substack{ \chi \shortmod q }}\mathcal{N}(\overline{\chi}, k) \mathcal{N}(\chi, k-1)) \\
=&  \sumstar_{\substack{ \chi \shortmod q }}\sum_{m \leq X}\frac {\chi(m)}{\sqrt{m}}\mathcal{N}(\overline{\chi}, k) \mathcal{N}(\chi,
k-1)+O(\frac {\phis(q)q^{1/2+4/10^{M}}\log q}{\sqrt{X}}),
\end{align*}
  where the last estimation above follows from \eqref{prodNbound}.  Applying \eqref{prodNbound} one more time, we see that the main term above
  equals
\begin{align*}
 & \sum_{\substack{ \chi \shortmod q }}\sum_{m \leq X}\frac {\chi(m)}{\sqrt{m}}\mathcal{N}(\overline{\chi}, k) \mathcal{N}(\chi,
 k-1)+O(\sqrt{X}q^{4/10^{M}}) \\
=& \phis(q) \sum_{a} \sum_{b} \sum_{\substack{n \leq X \\ an \equiv b \bmod q}}\frac {x_a y_b}{\sqrt{abn}}+O(\sqrt{X}q^{4/10^{M}}),
\end{align*}
  where we write for simplicity
\begin{align*}
 {\mathcal N}(\chi, k-1)= \sum_{a  \leq q^{2/10^{M}}} \frac{x_a}{\sqrt{a}} \chi(a), \quad \mathcal{N}(\overline{\chi}, k) = \sum_{b  \leq
 q^{2/10^{M}}} \frac{y_b}{\sqrt{b}}\overline{\chi}(b).
\end{align*}
  We now consider the contribution from the terms $am=b+l q$ with $l \geq 1$ above (note that as $b <q$, we can not have $b > am$ in our
  case). As this implies that $l \leq  q^{2/10^{M}}X/q$, we deduce together with the observation that $x_a, y_b \ll 1$ that the total
  contribution from these terms is
\begin{align*}
 \ll & \phi(q)  \sum_{b  \leq q^{2/10^{M}}}  \sum_{l \leq q^{2/10^{M}}X/q}\frac {1}{\sqrt{bql}} \ll \sqrt{X}q^{2/10^{M}}.
\end{align*}

   We now set $X=q^{1+1/10^{M-1}}$ to see that we can ignore the contributions from various error terms above to deduce that
\begin{align*}
& \sumstar_{\substack{ \chi \shortmod q }}L(\tfrac{1}{2},\chi) \mathcal{N}(\overline{\chi}, k) \mathcal{N}(\chi, k-1)
\gg \phis(q) \sum_{a} \sum_{b} \sum_{\substack{m \leq X \\ am = b }}\frac {x_a y_b}{\sqrt{abm}}
= \phis(q) \sum_{b} \frac {y_b}{b} \sum_{\substack{a, m \\ am = b }}x_a=\phis(q) \sum_{b} \frac {y_b}{b} \sum_{\substack{a | b }}x_a,
\end{align*}
  where the last equality above follows from the observation that $b \leq q^{2/10^{M}}<X$.

   Notice that
\begin{align}
\label{sumab}
\begin{split}
\sum_{b} \frac {y_b}{b} \sum_{\substack{a | b }}x_a=& \prod^R_{j=1}\Big ( \sum_{n_j} \frac{1}{n_j} \frac{k^{\Omega(n_j)}}{w(n_j)}
b_j(n_j)\sum_{n'_j|n_j} \frac{(k-1)^{\Omega(n'_j)}}{w(n'_j)}  b_j(n'_j) \Big ) \\
=& \prod^R_{j=1}\Big ( \sum_{n_j} \frac{1}{n_j} \frac{k^{\Omega(n_j)}}{w(n_j)}  b_j(n_j)\sum_{n'_j|n_j}  \frac{(k-1)^{\Omega(n'_j)}}{w(n'_j)}
\Big ),
\end{split}
\end{align}
  where the last equality above follows by noting that $b_j(n_j)=1$ implies that $b_j(n'_j)=1$ for all $n'_j|n_j$.

  We consider the sum above over $n_j$ for a fixed $1 \leq j \leq R$ in \eqref{sumab}. Note that the factor $b_j(n_j)$ restricts $n_j$ to have
  all prime factors in $P_j$ such that $\Omega(n_j) \leq \ell_j$. If we remove the restriction on $\Omega(n_j)$, then the sum becomes
\begin{align}
\label{6.02}
\begin{split}
& \prod_{\substack{p\in P_j }} \Big( \sum_{i=0}^{\infty} \frac{1}{p^i} \frac{k^{i}}{i!}\Big ( \sum_{l=0}^{i}
\frac{(k-1)^{l}}{l!} \Big ) \Big)
= \prod_{\substack{p\in P_j }}\Big (1+ \frac {k^2}p+O(\frac 1{p^2}) \Big ).
\end{split}
\end{align}

   On the other hand, using Rankin's trick by noticing that $2^{\Omega(n_j)-\ell_j}\ge 1$ if $\Omega(n_j) > \ell_j$,  we see that the error
   introduced this way does not exceed
\begin{align*}
\begin{split}
 & \Big ( \sum_{n_j} \frac{1}{n_j} \frac{k^{\Omega(n_j)}}{w(n_j)}2^{\Omega(n_j)-\ell_j}  \sum_{n'_j|n_j}  \frac{(1-k)^{\Omega(n'_j)}}{w(n'_j)}
 \Big ) \\
\le & 2^{-\ell_j} \prod_{\substack{p\in P_j }} \Big( \sum_{i=0}^{\infty} \frac{1}{p^i} \frac{(2k)^{i}}{i!}\Big ( \sum_{l=0}^{i}
\frac{(1-k)^{l}}{l!} \Big ) \Big) \\
\le & 2^{-\ell_j} \prod_{\substack{p\in P_j }} \Big( 1+ \frac {2k(2-2k)}p+O(\frac 1{p^2})\Big) \\
\leq & 2^{-\ell_j/2}\prod_{\substack{p\in P_j }}\Big (1+ \frac {k^2}p+O(\frac 1{p^2}) \Big ),
\end{split}
\end{align*}
  where the last estimation above follows by taking $N$ large enough so that we deduce from Lemma \ref{RS} that
\begin{align}
\label{boundsforsumoverp}
   \sum_{p \in P_j}\frac 1{p} \leq \frac 1N \ell_j.
\end{align}

  We then deduce from this, \eqref{6.02} and Lemma \ref{RS} that we have
\begin{align*}
& \sumstar_{\substack{ \chi \shortmod q }}L(\tfrac{1}{2},\chi) \mathcal{N}(\overline{\chi}, k) \mathcal{N}(\chi, k-1)
\gg \phis(q) \prod^R_{j=1}\Big (1+O( 2^{-\ell_j/2})\Big )\prod_{\substack{p\in P_j }}\Big (1+ \frac {k^2}p+O(\frac 1{p^2}) \Big ) \gg
\phis(q)(\log q)^{k^2}.
\end{align*}
 This completes the proof of the proposition.

\section{Proof of Proposition \ref{Prop5}}
\label{sec 5}

  As the sum over $e^{-\ell_j/2}$ converges, we see that it suffices to show that for a fixed integer $v$ such that $1 \leq v \leq R-1$, we have
\begin{align*}
\begin{split}
  \sumstar_{\substack{ \chi \shortmod q }}|L(\tfrac{1}{2},\chi)|^2 \Big (\prod^v_{j=1}|\mathcal{N}_j(\chi, k-1)|^2 \Big )|{\mathcal
 Q}_{v+1}(\chi, k)|^{2} \ll \phis(q)e^{-\ell_{v+1}/2}(\log q)^{k^2}.
\end{split}
\end{align*}

   We define the function $p_{v+1}(n)$ such that $p_{v+1}(n)=0$ or $1$, and we have $p_{v+1}(n)=1$ if and only if $n$ is composed of exactly $r_k\ell_{v+1}$ primes (counted with multiplicity), all from the interval $P_{v+1}$. We use this together with the notations in Section \ref{sec 4} to write that
\begin{align}
\label{Pexpression}
  {\mathcal P}_{v+1}(\chi)^{r_k\ell_{v+1}} =&  \sum_{ \substack{ n_{v+1}}} \frac{1}{\sqrt{n_{v+1}}}\frac{(r_k\ell_{v+1})!
  }{w(n_{v+1})}\chi(n_{v+1})p_{v+1}(n_{v+1}),
\end{align}
  where we recall that $r_k=\lceil 1+1/k \rceil+1$ when $0<k<1$. We then apply the above to write
$$ \Big (\prod^v_{j=1}|\mathcal{N}_j(\chi, k-1)|^2 \Big )|{\mathcal
 Q}_{v+1}(\chi, k)|^{2} = \Big( \frac{12  }{\ell_{v+1}}\Big)^{2r_k\ell_{v+1}}((r_k\ell_{v+1})!)^2 \sum_{a,b \leq q^{2r_k/10^{M}}} \frac{u_a u_b}{\sqrt{ab}}\chi(a)\overline{\chi}(b).$$
 Here we observe that $\prod^v_{j=1}|\mathcal{N}_j(\chi, k-1)| \cdot |{\mathcal
 Q}_{v+1}(\chi, k)|$ is a short Dirichlet polynomial whose length is at most
$$  q^{1/\ell_1+ \ldots +1/\ell_v+r_k/\ell_{v+1}} < q^{2r_k/10^{M}}.$$
Note also that we have $u_a, u_b \leq 1$ for all $a, b$.

 We evaluate the sum of $\Big (\prod^v_{j=1}|\mathcal{N}_j(\chi, k-1)|^2 \Big )|{\mathcal
 Q}_{v+1}(\chi, k)|^{2}$ over all primitive characters over $q$ by splitting the sum into
 sums over even and odd characters separately. As the treatments are similar, we only consider the sum over even characters here.  From
 \eqref{lsquareapprox} and Lemma \ref{lem:sumoverevenodd}, we have that
\begin{align}
\label{LNsquaresum}
\begin{split}
& {\sum_{\chi}}^{(e)} |L(1/2, \chi)|^2\Big (\prod^v_{j=1}|\mathcal{N}_j(\chi, k-1)|^2 \Big )|{\mathcal
 Q}_{v+1}(\chi, k)|^{2} \\
=& 2\Big( \frac{12  }{\ell_{v+1}}\Big)^{2r_k\ell_{v+1}}((r_k\ell_{v+1})!)^2 \sum_{a,b \leq q^{2r_k/10^{M}}} \frac{u_a u_b}{\sqrt{ab}} \sum_{m, n} \frac{1}{\sqrt{mn}} \Wf_{\af} \left(\frac {\pi mn}{q}\right)
{\sum_{\chi}}^{(e)} \chi(ma) \overline{\chi}(nb) \\
=& \phis(q) \Big( \frac{12  }{\ell_{v+1}}\Big)^{2r_k\ell_{v+1}}((r_k\ell_{v+1})!)^2 \sum_{a,b \leq q^{2r_k/10^{M}}} \frac{x_a x_b}{\sqrt{ab}}  \sum_{\substack{m,n \\ (mn, q)=1 \\ ma \equiv \pm nb \,{\rm mod} \, q}}
\frac{1}{\sqrt{mn}} \Wf_{\af} \left(\frac {\pi mn}{q}\right) \\
& + O\left( \Big( \frac{12  }{\ell_{v+1}}\Big)^{2r_k\ell_{v+1}}((r_k\ell_{v+1})!)^2\sum_{a,b \leq q^{2r_k/10^{M}}} \frac{1}{\sqrt{ab}}  \sum_{m,n}
\frac{1}{\sqrt{mn}} \Wf_{\af} \left(\frac {\pi mn}{q}\right)  \right).
\end{split}
\end{align}

 We now apply \eqref{Stirling} and the definition of $\ell_{v+1}$ to see that
\begin{align}
\label{factorbound}
 & \Big( \frac{12 }{\ell_{v+1}} \Big)^{2r_k\ell_{v+1}}((r_k \ell_{v+1})!)^2
\leq (r_k\ell_{v+1})^2 \Big( \frac{12 r_k }{e } \Big)^{2r_k\ell_{v+1}} \ll q^{\varepsilon}.
\end{align}

   It follows from this that the error term in \eqref{LNsquaresum} is
$$ \ll q^{2r_k/10^{M}+\varepsilon} \sum_{d} \frac{d^{\epsilon}}{\sqrt{d}} \Wf_{\af}\left(\frac {\pi d}{q}\right)\ll q^{1-\varepsilon}.$$

We next estimate the contribution of the terms $ma \neq nb$ in the last expression of \eqref{LNsquaresum}. By the rapid decay of $\Wf_{\af}$
given in Lemma \ref{PropDirpoly}, we may assume that $mn \leq q^{1 + \varepsilon}$. Using $\Wf_{\af} \left(\frac {\pi mn}{q}\right) \ll 1$ and \eqref{factorbound}, we see that these terms contribute
$$ \ll q^{1+\varepsilon} \sum_{a,b \leq q^{2r_k/10^{M}}} \frac{1}{\sqrt{ab}} \sum_{\substack{m, n \\ mn \leq q^{1 + \varepsilon} \\ q | ma \pm nb \\ ma \pm nb \neq
0 }} \frac{1}{\sqrt{mn}}.$$

  To estimate the expression above, we may consider the case $q|ma+nb$ without loss of generality. We may also assume that $ma \geq nb$
  so that on writing $ma+nb=ql$, we have that $ql \leq 2ma \leq 2q^{1 + \varepsilon}q^{2r_k/10^{M}}$ which implies that $l \leq
  2q^{2r_k/10^{M}+\varepsilon}$. Moreover, we have that $1/\sqrt{ma} \ll 1/\sqrt{ql}$ and that $mn \leq q^{1 + \varepsilon}$ implies that  $mnab
  \leq q^{1 + \varepsilon}ab \leq q^{1 + \varepsilon}q^{4r_k/10^{M}}$, so that we have $n \leq nb \leq q^{1/2 + \varepsilon}q^{2r_k/10^{M}}$. It
  follows that
\begin{align*}
  & q^{1+\varepsilon} \sum_{a,b \leq q^{2r_k/10^{M}}} \frac{1}{\sqrt{ab}} \sum_{\substack{m, n \\ mn \leq q^{1 + \varepsilon} \\ q | ma + nb \\ ma + nb \neq 0 \\
  ma \geq nb}} \frac{1}{\sqrt{mn}} \ll  q^{1+\varepsilon} \sum_{b \leq q^{2r_k/10^{M}}} \frac{1}{\sqrt{b}} \sum_{\substack{n \\ n  \leq q^{1/2 +
  \varepsilon}q^{2r_k/10^{M}}}} \frac{1}{\sqrt{n}}\sum_{\substack{l \leq 2q^{2r_k/10^{M}+\varepsilon}}} \frac{1}{\sqrt{ql}}  \ll q^{1-\varepsilon}.
\end{align*}

 It remains to consider the terms $ma = nb$ in the last expression of \eqref{LNsquaresum}. We write $m = \frac{\alpha b}{(a,b)}, n =
 \frac{\alpha a}{(a,b)}$ and apply \eqref{factorbound} to see that these terms are
\begin{align}
\label{mainterm}
\ll & \phis(q)(r_k\ell_{v+1})^2 \Big( \frac{12 r_k }{e } \Big)^{2r_k\ell_{v+1}} \sum_{a,b \leq q^{2r_k/10^{M}}} \frac{(a,b)}{ab} u_a u_b  \sum_{(\alpha, q)=1} \frac{1}{\alpha} \Wf_{\af} \left(\frac{\pi\alpha^2
ab}{q(a,b)^2}\right).
\end{align}

  To evaluate the last sum above, we set $X=q(a,b)^2/(\pi ab)$ and apply the definition of $\Wf_{\af}(x)$ given in Lemma \ref{PropDirpoly} to
  see that
\begin{align*}
& \sum_{(\alpha, q)=1} \frac{1}{\alpha} \Wf_{\af} \left(\frac{\alpha^2}{X}\right)= \frac{1}{2\pi i} \int\limits_{(c)}
\frac{\Gamma\left(\frac{1}{4} + \frac{s + \af}{2}\right)^2}{\Gamma\left(\frac{1}{4} + \frac{\af}{2} \right)^2 }\zeta(1+2s) (1-q^{-1-2s})X^{s}
\frac{ds}{s}.
\end{align*}
  We evaluate the integral above by shifting the line of integration to $\Re(s)=-1/4+\varepsilon$. We encounter a double pole at $s=0$ in the
  process. The integration on the new line can be estimated trivially using the convexity bound for $\zeta(s)$  (see \cite[Exercise 3, p.
  100]{iwakow}) as
\begin{align*}
\begin{split}
  \zeta(s) \ll & \left( 1+|s| \right)^{\frac {1-\Re(s)}{2}+\varepsilon}, \quad 0 \leq \Re(s) \leq 1,
\end{split}
\end{align*}
   and the rapid decay of $\Gamma(s)$ when $|\Im(s)| \rightarrow \infty$. We also evaluate the corresponding residue to see that
\begin{align}
\label{alphasum}
& \sum_{(\alpha, q)=1} \frac{1}{\alpha} \Wf_{\af} \left(\frac{\alpha^2}{X}\right)= C_1(q)\log X+C_2(q)+O(X^{-1/4+\varepsilon}),
\end{align}
   where $C_1(q), C_2(q)$ are some constants depending on $q$, satisfying $C_1(q), C_2(q) \ll 1$.

  We apply \eqref{alphasum} to evaluate \eqref{mainterm} to see that we may ignore the contribution of the error term from \eqref{alphasum} so
  that the expression in \eqref{mainterm} is
\begin{align*}
& \ll \phis(q) (r_k\ell_{v+1})^2 \Big( \frac{12 r_k }{e } \Big)^{2r_k\ell_{v+1}} \sum_{a,b \leq q^{2r_k/10^{M}}} \frac{(a,b)}{ab} u_a u_b \Big ( C_1(q)( \log q+2 \log (a, b)-\log a -\log b-\log \pi)+C_2(q) \Big ).
\end{align*}

  As the estimations are similar, it suffices to give an estimation on the sum
\begin{align}
\label{sumoverlog}
\begin{split}
& \phis(q)(r_k\ell_{v+1})^2 \Big( \frac{12 r_k }{e } \Big)^{2r_k\ell_{v+1}} \sum_{a,b } \frac{(a,b)}{ab} u_a u_b \log a \\
=& \phis(q)(r_k\ell_{v+1})^2 \Big( \frac{12 r_k }{e } \Big)^{2r_k\ell_{v+1}} \sum_{p \in \bigcup^{v+1}_{j=1} P_j}\sum_{l_1 \geq 1, l_2 \geq 0}\frac {l_1 \log p}{p^{l_1+l_2-\min (l_1, l_2)}}\frac {(k-1)^{l_1+l_2}}{l_1!l_2!}\sum_{\substack{ a,b \\ (ab, p)=1} } \frac{(a,b) u_{p^{l_1}a} u_{p^{l_2}b}}{ab} .
\end{split}
\end{align}
   We now estimate the sum last sum above for fixed $p=p_1, l_1, l_2$. Without loss of generality, we may assume that $p=p_1 \in P_1$. We then define for $(n_1n'_1, p_1)=1$,
\begin{equation*}
 v_{n_1} = \frac{1}{n_1} \frac{(k-1)^{\Omega(n_1)}}{w(n_1)}  b_1(n_1p_1^{l_1}), \quad v_{n'_1} = \frac{1}{n'_1} \frac{(k-1)^{\Omega(n'_1)}}{w(n'_1)}  b_1(n'_1p_1^{l_2}).
\end{equation*}
  For $2 \leq j \leq v$,
\begin{equation*}
 v_{n_j} = \frac{1}{n_j} \frac{(k-1)^{\Omega(n_j)}}{w(n_j)}  b_j(n_j), \quad  v_{n'_j} = \frac{1}{n'_j} \frac{(k-1)^{\Omega(n'_j)}}{w(n'_j)}  b_j(n'_j).
\end{equation*}
  Also,
\begin{equation*}
 v_{n_{v+1}} = \frac{1}{n_{v+1}} \frac{1}{w(n_{v+1})}  p_{v+1}(n_{v+1}), \quad  v_{n'_{v+1}} = \frac{1}{n'_{v+1}} \frac{1}{w(n'_{v+1})}  p_{v+1}(n'_{v+1}).
\end{equation*}

  Then one checks that
\begin{align}
\label{doublesum}
 \sum_{\substack{ a,b \\ (ab, p_1)=1} } \frac{(a,b) u_{p^{l_1}a} u_{p^{l_2}b}}{ab}=\prod^{v+1}_{j=1}\Big ( \sum_{\substack{n_j, n'_j \\ (n_1n'_1, p_1)=1}}(n_j, n'_j)v_{n_j}v_{n'_j} \Big ).
\end{align}

   As in the proof of Proposition \ref{Prop4}, when $\max (l_1, l_2) \leq \ell_1/2$, we remove the restriction of $b_1(n_1)$ on $\Omega_1(n_1)$ and $b_1(n'_1)$ on $\Omega_1(n'_1)$ to see that the last sum in \eqref{doublesum} becomes
\begin{align*}
  & \sum_{\substack{n_1, n'_1\\ (n_1n'_1, p_1)=1}}(n_1, n'_1)\frac{1}{n_1} \frac{(k-1)^{\Omega(n_1)}}{w(n_1)} \frac{1}{n'_1} \frac{(k-1)^{\Omega(n'_1)}}{w(n'_1)} =
  \prod_{\substack{p \in P_1 \\ p \neq p_1}}\Big ( 1+\frac{2(k-1)+(k-1)^2}{p} +O(\frac 1{p^2})  \Big ) \\
\ll &  \exp (\sum_{p \in P_1}\frac{k^2-1}{p}+O(\sum_{p \in P_1}\frac 1{p^2})).
\end{align*}

  Further, we notice that in this case we have $2^{\Omega(n)-\ell_1/2}\ge 1$ if $\Omega(n)+\max (l_1, l_2) \geq \ell_1$.  Thus, we apply Rankin's
  trick to see that the error introduced this way is
\begin{align*}
 \ll &  2^{-\ell_1/2} \sum_{n_1, n'_1}\frac {(n_1, n'_1)}{n_1n'_1} \frac{(1-k)^{\Omega(n_1)}}{w(n_1)} \frac{(1-k)^{\Omega(n'_1)}2^{\Omega(n'_1)}}{w(n'_1)}
 =2^{-\ell_1/2} \prod_{p \in P_1}\Big ( 1+\frac{3(1-k)+2(1-k)^2}{p} +O(\frac 1{p^2})  \Big ).
\end{align*}
   We may take $N$ large enough so that it follows from \eqref{boundsforsumoverp} that when $\max (l_1, l_2) \leq \ell_1/2$, the error is
\begin{align*}
 \ll &  2^{-\ell_1/4} \exp (\sum_{p \in P_1}\frac{k^2-1}{p}+O(\sum_{p \in P_1}\frac 1{p^2})).
\end{align*}

   On the other hand, when $\max (l_1, l_2) > \ell_1/2$, we apply \eqref{boundsforsumoverp} again to see that
\begin{align*}
 &  \sum_{\substack{n_1, n'_1 \\ (n_1n'_1, p_1)=1}}(n_1, n'_1)v_{n_1}v_{n'_1} \ll  \sum_{\substack{n_1, n'_1 \\ (n_1n'_1, p_1)=1}}\frac{(n_1, n'_1)}{n_1n'_1}
 \frac{(1-k)^{\Omega(n_1)}}{w(n_1)}  \frac{(1-k)^{\Omega(n'_1)}}{w(n'_1)} \ll 2^{\ell_1/6}  \exp (\sum_{p \in P_1}\frac{k^2-1}{p}+O(\sum_{p \in P_1}\frac
 1{p^2})).
\end{align*}
   We then deduce that in this case
\begin{align*}
 &  \frac {l_1 \log p_1}{p^{l_1+l_2-\min (l_1, l_2)}_1}\frac {(k-1)^{l_1+l_2}}{l_1!l_2!}\sum_{\substack{n_1, n'_1 \\ (n_1n'_1, p_1)=1}}(n_1, n'_1)v_{n_1}v_{n'_1} \ll \frac {l_1
 \log p_1}{p^{\max (l_1, l_2)}_1} \frac {1}{l_1!l_2!} 2^{\ell_1/6}  \exp (\sum_{p \in P_1}\frac{k^2-1}{p}+O(\sum_{p \in P_1}\frac 1{p^2})) \\
 \ll & \frac {l_1 \log p_1}{p^{l_1/2+\max (l_1, l_2)/2}_1} \frac {1}{l_1!l_2!} 2^{\ell_1/6} \exp (\sum_{p \in P_1}\frac{k^2-1}{p}+O(\sum_{p \in
 P_1}\frac 1{p^2})) \\
 \ll & \frac {l_1 \log p_1}{p^{l_1/2}_1}\frac {1}{l_1!l_2!} 2^{-\ell_1/12}   \exp (\sum_{p \in  P_1}\frac{k^2-1}{p_1}+O(\sum_{p \in  P_1}\frac 1{p^2})).
\end{align*}

   Similar estimations carry over to the sums over $n_j, n'_j$ for $2 \leq j \leq v$ in \eqref{doublesum}. To treat the sum over $n_{v+1}, n'_{v+1}$,  we apply Rankin's trick again to see that the sum is
\begin{align*}
\ \ll &  (12r_k)^{-2r_k\ell_{v+1}} \sum_{\substack{n_{v+1}, n'_{v+1} \\ p |
n_{v+1}n'_{v+1} \implies  p\in P_{v+1} } }\frac {(n_{v+1}, n'_{v+1})}{n_{v+1}n'_{v+1}} \frac{(12r_k)^{\Omega(n_{v+1})}}{w(n_{v+1})} \frac{A^{\Omega(n'_{v+1})}}{w(n'_{v+1})}.
\end{align*}
  By taking $N$ large enough, we deduce from this that
\begin{align*}
\begin{split}
&  (r_k\ell_{v+1})^2 \Big( \frac{12 r_k }{e } \Big)^{2r_k\ell_{v+1}} \sum_{\substack{n_{v+1}, n'_{v+1} \\ (n_{v+1}n'_{v+1}, p_1)=1}}(n_{v+1}, n'_{v+1})v_{n_{v+1}}v_{n'_{v+1}}
\ll  e^{-\ell_{v+1}}\exp (\sum_{p \in P_{v+1}}\frac{k^2-1}{p}+O(\sum_{p \in P_{v+1}}\frac 1{p^2})).
\end{split}
\end{align*}

   It follows from the above discussions that we have
\begin{align}
\label{sumpbound}
\begin{split}
&  (r_k\ell_{v+1})^2 \Big( \frac{12 r_k }{e } \Big)^{2r_k\ell_{v+1}} \sum_{l_1 \geq 1, l_2 \geq 0}\frac {l_1 \log p_1}{p^{l_1+l_2-\min (l_1, l_2)}_1}\frac {(k-1)^{l_1+l_2}}{l_1!l_2!}\sum_{\substack{ a,b \\ (ab, p_1)=1} } \frac{u_{p^{l_1}_1a}
u_{p^{l_2}_1b}}{ab} \\
 \ll & e^{-\ell_{v+1}}\prod^v_{j=1}\big(1+O(2^{-\ell_j/12})\big)\exp (\sum_{p \in \bigcup^{v+1}_{j=1} P_j}\frac{k^2-1}{p}+O(\sum_{p \in \bigcup^{v+1}_{j=1} P_j}\frac 1{p^2})) \times \Big ( \frac{\log p_1}{p_1} + O(\frac
 {\log p_1}{p^2_1}) \Big ).
\end{split}
\end{align}

  We now apply \eqref{boundsforsumoverp} and the observation $\ell_j > \ell^2_{j+1}>2\ell_{j+1}$ to see that
\begin{align}
\label{sumpbound1}
 \sum_{p \in \bigcup^R_{j=v+2} P_j}\frac{1-k^2}{p} \leq \sum^{R}_{j=v+2}\sum_{p \in P_j}\frac 1{p} \leq \frac 1N \sum^{R}_{j=v+2} \ell_j \leq \frac {\ell_{v+2}}{N} \sum^{\infty}_{j=0} \frac 1{2^j} \leq \frac {\ell_{v+1}}{N}.
\end{align}

   It follows from \eqref{sumpbound1} that the last expression in \eqref{sumpbound} is
\begin{align*}
\begin{split}
 \ll &  e^{-\ell_{v+1}/2}\exp (\sum_{p \in \bigcup^{R}_{j=1} P_j}\frac{k^2-1}{p}+O(\sum_{p \in \bigcup^{R}_{j=1} P_j}\frac 1{p^2})) \times \Big ( \frac{\log p_1}{p_1} + O(\frac
 {\log p_1}{p^2_1}) \Big ).
\end{split}
\end{align*}

 We then conclude from this, \eqref{sumoverlog} and Lemma \ref{RS} that
\begin{align*}
\begin{split}
& {\sum_{\chi}}^{(e)} |L(1/2, \chi)|^2\Big (\prod^v_{j=1}|\mathcal{N}_j(\chi, k-1)|^2 \Big )|{\mathcal
 Q}_{v+1}(\chi, k)|^{2} \\
\ll &    \phis(q) e^{-\ell_{v+1}/2}\exp (\sum_{p \in \bigcup^{R}_{j=1} P_j}\frac{k^2-1}{p}+O(\sum_{p \in \bigcup^{R}_{j=1} P_j}\frac 1{p^2})) \times \sum_{p \in \bigcup^{v+1}_{j=1}P_j} \Big ( \frac{\log p}{p} + O(\frac
 {\log p}{p^2}) \Big ) \ll \phis(q) e^{-\ell_{v+1}/2}(\log q)^{k^2}.
\end{split}
\end{align*}

   This completes the proof of the proposition.

\section{Proof of Proposition \ref{Prop6}}

   As the proofs are similar, we shall only prove here that
\begin{align*}
  \sumstar_{\substack{ \chi \shortmod q }}\prod^R_{j=1}\big ( |{\mathcal N}_j(\chi, k)|^2+ |{\mathcal Q}_j(\chi,k)|^2 \big ) \ll \phis(q)(\log q)^{ k^2 }.
\end{align*}

   We first note that
\begin{align}
\label{upperboundprodofN}
\begin{split}
 & \sumstar_{\substack{ \chi \shortmod q }}\prod^R_{j=1}\big ( |{\mathcal N}_j(\chi, k)|^2+ |{\mathcal Q}_j(\chi,k)|^2 \big )
\leq  \sum_{\substack{ \chi \shortmod q }}  \prod^R_{j=1}\big ( |{\mathcal N}_j(\chi, k)|^2+|{\mathcal Q}_j(\chi,k)|^2 \big ).
\end{split}
\end{align}

  We shall take $M$ large enough so that we may deduce from \eqref{sumoverell} that
\begin{align*}
  (2 r_k+2) \sum^R_{j=1}\frac 1{\ell_j} \leq \frac {4(r_k+1)}{\ell_R} <1 .
\end{align*}

It follows from this, \eqref{5.1}, \eqref{Pexpression} and the orthogonality relation for characters modulo $q$ that only the
  diagonal terms in the last sum of \eqref{upperboundprodofN} survive. This implies that
\begin{align}
\label{maintermbound}
\begin{split}
 & \sum_{\substack{ \chi \shortmod q }}  \prod^R_{j=1}\big ( |{\mathcal N}_j(\chi, k)|^2+ |{\mathcal Q}_j(\chi,k)|^2 \big ) \\
\le & \phis(q) \prod^R_{j=1} \Big (  \sum_{n_j} \frac{k^{2\Omega(n_j)}}{n_j w^2(n_j)}  b_j(n_j)  + \Big( \frac{12} {\ell_j}\Big)^{2r_k\ell_j}((r_k\ell_j)!)^2 \sum_{ \substack{ \Omega(n_j) = r_k\ell_j \\ p|n_j \implies  p\in P_j}} \frac{1 }{n_j w^2(n_j)} \Big ).
\end{split}
\end{align}

  Arguing as before, we see that
\begin{align}
\label{sqinN}
\begin{split}
  \sum_{n_j} \frac{k^{2\Omega(n_j)}}{n_j w^2(n_j)}  b_j(n_j)=\Big(1+ O\big(2^{-\ell_j/2} \big ) \Big)
  \exp (\sum_{p \in P_j} \frac {k^2}{p}+ O(\sum_{p \in P_j} \frac {1}{p^2})).
\end{split}
\end{align}
   Note also that,
\begin{align*}
\begin{split}
 \sum_{ \substack{ \Omega(n_j) = r_k\ell_j \\ p|n_j \implies  p\in P_j}} \frac{1 }{n_j w^2(n_j)}  \leq  \frac 1{(r_k \ell_j)!} \Big (\sum_{p \in P_j}\frac 1{p} \Big )^{r_k\ell_j}.
\end{split}
\end{align*}

  Now, we apply \eqref{Stirling} and \eqref{boundsforsumoverp} to deduce from the above that by taking $M, N$ large enough,
\begin{align}
\label{Qest}
\begin{split}
 &  \Big( \frac{12} {\ell_j}\Big)^{2r_k\ell_j}((r_k\ell_j)!)^2 \sum_{ \substack{ \Omega(n_j) = r_k\ell_j \\ p|n_j \implies  p\in P_j}} \frac{1 }{n_j w^2(n_j)} \ll r_k\ell_j\Big( \frac{144 r_k }{e\ell_j} \Big)^{r_k\ell_j}\Big (\sum_{p \in P_j}\frac 1{p} \Big )^{r_k\ell_j} \\
\ll &  r_k\ell_j\Big( \frac{144 r_k }{e\ell_j} \Big)^{r_k\ell_j}e^{r_k\ell_j \log (2\ell_j/N)}
\ll  e^{-\ell_j} \exp (\sum_{p \in P_j} \frac {k^2}{p}+ O(\sum_{p \in P_j} \frac {1}{p^2})).
\end{split}
\end{align}
  Using \eqref{sqinN} and \eqref{Qest} in \eqref{maintermbound} and then applying Lemma \ref{RS}, we readily deduce the assertion of the proposition.

\vspace*{.5cm}

\noindent{\bf Acknowledgments.} P. G. is supported in part by NSFC grant 11871082.

\bibliography{biblio}
\bibliographystyle{amsxport}

\vspace*{.5cm}

\end{document}